\documentclass[12pt, a4paper, reqno]{amsart}
\usepackage{amsmath}
\usepackage{amsfonts}
\usepackage{amssymb}
\usepackage{amsthm}
\usepackage{url}

\newcommand{\R}{\mathbb{R}}
\newcommand{\Rn}{\mathbb{R}^n}

\def\mvint_#1{\mathchoice
          {\mathop{\vrule width 6pt height 3 pt depth -2.5pt
                  \kern -8pt \intop}\nolimits_{#1}}%
          {\mathop{\vrule width 5pt height 3 pt depth -2.6pt
                  \kern -6pt \intop}\nolimits_{#1}}%
          {\mathop{\vrule width 5pt height 3 pt depth -2.6pt
                  \kern -6pt \intop}\nolimits_{#1}}%
          {\mathop{\vrule width 5pt height 3 pt depth -2.6pt
                  \kern -6pt \intop}\nolimits_{#1}}}
\def\sqr#1#2{{\vcenter{\hrule height.#2pt\hbox{\vrule
     width.#2pt height#1pt\kern#1pt\vrule width.#2pt}\hrule height.#2pt}}}

\def\XXint#1#2#3{{\setbox0=\hbox{$#1{#2#3}{\int}$}
     \vcenter{\hbox{$#2#3$}}\kern-.5\wd0}}

\theoremstyle{plain}
\newtheorem{theorem}[equation]{Theorem}
\newtheorem{lemma}[equation]{Lemma}
\newtheorem{corollary}[equation]{Corollary}

\numberwithin{equation}{section}

\theoremstyle{definition}

\newtheorem{example}[equation]{Example}

\theoremstyle{remark}
\newtheorem{remark}[equation]{Remark}
\title{Weak $L^{\infty}$ and $BMO$ in metric spaces}
\author{Daniel Aalto}
\date{\today}

\begin{document}



\begin{abstract}
Bennett, DeVore and Sharpley introduced the space weak $L^{\infty}$ in
1981 and studied its relationship with functions of bounded mean
oscillation. Here we characterize the weak
$L^{\infty}$ in measure spaces 
without using the decreasing rearrangement of a function.
Instead, we use exponential estimates for the distribution function. 
In addition, we consider a localized version of the characterization
that leads to a new characterization of BMO. 
\end{abstract}

\maketitle


\section{Introduction}

Bennett DeVore and Sharpley introduced the space weak $L^{\infty}$ in \cite{BDS}.
The definition of the weak $L^{\infty}$ is based on decreasing
rearrangements 
(see also the generalizations in \cite{BMR}, \cite{MaPi} and \cite{MiPu}). 
We give here more geometric characterizations of the weak $L^{\infty}$
by analyzing the decay of the distribution functions. 
The main result in \cite{BDS} states that the weak $L^{\infty}(Q)$, 
where $Q$ is a Euclidean cube, 
is the rearrangement invariant hull of BMO (see also \cite{BS}
and \cite{Ko}). 
We show by an example
that the weak $L^{\infty}$ is not the rearrangement invariant hull of BMO
in general. 

We
localize the geometric characterization of the weak $L^{\infty}$ to obtain 
a new characterization of BMO, too. To show that our argument is based on
a general principle,  
we study the characterizations in doubling measure spaces. Indeed,
there has been  
a considerable interest in extending classical results 
in harmonic analysis to the metric setting, 
see e.g. \cite{CW} and \cite{H}. 
The most important ingredient in our argument is a 
Calder\'on-Zygmund type covering lemma, which may be useful also
elsewhere. 
There are several versions 
of this type of covering lemmas in the literature, see \cite{CW} and \cite{MP}, but the version 
presented here seem not to follow immediately from any of them.

{\bf Acknowledgements.}
The author is supported by the Finnish Academy of Science and Letters,
the Vilho, Yrj\"o and Kalle V\"ais\"al\"a Foundation.

\section{A Characterisation of The Weak $L^{\infty}$}

Let $(X,\mu)$ be a measure space. 
The distribution function of a real-valued function $f$ defined in $X$
is the function 
$d_{f,\mu}:[0,\infty)\to [0,\infty]$ defined by the formula
$$
d_{f,\mu}(\lambda) = \mu (\{x\in X: |f(x)|>\lambda\}).
$$
If $d_{f,\mu}$ is invertible, its inverse function $f^*$ is called the
decreasing rearrangement of $f$. More generally we define the
decreasing rearrangement of $f$ to be
the decreasing function $f^*:(0,\infty)\to [0,\infty]$ defined by
$$
f^*(t)=\inf \{ \lambda : d_{f,\mu}(\lambda)\leq t\}.
$$
Here we use the convention $\inf \emptyset = \infty$. The decreasing
rearrangement is unique, right-continuous, homogeneous, sublinear and
satisfies
$f^*(d_{f,\mu}(\lambda))\leq \lambda$, whenever $d_{f,\mu}(\lambda)$
is finite, and 
$d_{f,\mu}(f^*(t))\leq t$, whenever $f^*(t)$ is finite (c.f. \cite{BS}). 
The preceding 
inequalities can be taken as a definition of $f^*$ as well (as in
\cite{R}).
We say that $f$ defined in $(X,\mu)$
and $g$ defined in $(Y,\nu)$ 
are equimeasurable functions if $d_{f,\mu}=d_{g,\nu}$. 
Observe that $f$ and $f^*$ are
equimeasurable.

The following Cavalieri principle is useful for us: if $f$ is
$\mu$-measurable, then for $0<p<\infty$ we have
\begin{equation}
\label{cava1}
\nonumber
\int_X|f|^pd\mu = p \int_0^{\infty}\lambda^{p-1}d_{f,\mu}(\lambda)d\lambda.
\end{equation}

We define the maximal function $f^{**}$ of $f$ by
$$
f^{**}(t) = \frac{1}{t}\int_0^tf^*(s)ds.
$$
The space weak $L^{\infty}(X)$ or $L^{\infty}_w(X)$ is the collection
of all $f$ so that $f^*$ is finite everywhere and
$$
\sup_{t>0}(f^{**}(t)-f^*(t))<\infty.
$$
Observe that the weak $L^{\infty}(X)$ is rearrangement invariant
and that it fails to be a vectorspace. The next theorem gives a
characterization of the weak $L^{\infty}(X)$.
\begin{theorem}
\label{theorem1}
Let $f\in L^1_{loc}(X)$ so that
$d_{f,\mu}$ is not identically infinite.
Then the following conditions are equivalent:
\begin{itemize}
\item[(i)] The rearrangement of $f$ is
  finite for all $t>0$ and there exists $M\geq 0$ independent of $t$ for which
$$
f^{**}(t)-f^*(t)\leq M.
$$
\item[(ii)] There exist 
constants $\alpha \geq 0$ and $M\geq 0$ so that
for all $\lambda > \alpha$, we have $d_{f,\mu}(\lambda)<\infty$ and
$$
\int_{\{x\in X:|f(x)|>\lambda \} }|f|d\mu 
\leq 
(\lambda + M)d_{f,\mu}(\lambda).
$$
\item[(iii)] There exist 
constants $\alpha \geq 0$ and $M\geq 0$ so that
for all $\lambda > \alpha$, we have $d_{f,\mu}(\lambda)<\infty$ and
$$
\int_{\lambda }^{\infty}d_{f,\mu}(s)ds\leq Md_{f,\mu}(\lambda).
$$
\item[(iv)]
There exist constants $c_1,c_2>0$ so that
$$
d_{f,\mu}(\lambda_2)\leq c_1d_{f,\mu}(\lambda_1)e^{c_2(\lambda_1-\lambda_2)}
$$
for all $\lambda_2>\lambda_1\geq 0$.
\end{itemize}
\end{theorem} 

\begin{proof}
Observe that, as a consequence of Cavalieri's principle, 
if $f^*(t)$ is finite, we have
\begin{eqnarray}
\nonumber
\int_0^tf^*(s)ds
&=&
\int_0^{f^*(t)}td\lambda
+
\int_{f^*(t)}^{\infty}|\{s>0:f^*(s)>\lambda\}|d\lambda
\\
\nonumber
&=&
tf^*(t)+\int_{f^*(t)}^{\infty}d_{f,\mu}(\lambda)d\lambda.
\end{eqnarray}
This implies that
\begin{equation}
\label{id1}
f^{**}(t)-f^{*}(t)=\frac{1}{t}\int_{f^*(t)}^{\infty}d_{f,\mu}(\lambda)d\lambda.
\end{equation}

Let us assume the condition (i) and show that (ii) follows.
For $f\in L^{\infty}_w(X)$ we define
$$
\alpha = \lim_{t\to \infty}f^*(t).
$$
The limit exists since $f^*$ is decreasing and bounded below. Pick any
$s> \alpha$. If $d_{f,\mu}(s)=0$, there is nothing to prove.
Hence, we may assume $d_{f,\mu}(s)>0$. 
Using Cavalieri's principle we get
\begin{eqnarray}
\label{cava2}
\int_{\{x\in X:|f(x)|>s\}}|f|d\mu
&=&
\int_s^{\infty}d_{f,\mu}(\lambda)d\lambda
+
sd_{f,\mu}(s).
\end{eqnarray}
Since
$f^*(d_{f,\mu}(s))
\leq
s$, by (\ref{id1}) we have
\begin{eqnarray}
\nonumber
\int_{\{x\in X:|f(x)|>s\}}|f|d\mu
&\leq&
\int_{f^*(d_{f,\mu}(s))}^{\infty}d_{f,\mu}(\lambda)d\lambda + sd_{f,\mu}(s)
\\
\nonumber
&=&
\left(f^{**}(d_{f,\mu}(s))-f^{*}(d_{f,\mu}(s)) + s\right)d_{f,\mu}(s)
\\
\nonumber
&\leq&
(M+s)d_{f,\mu}(s).
\end{eqnarray}

Assume then that 
the condition (ii) holds.
Take any $t>0$.
Since
$$
d_{f,\mu}(f^*(t))
\leq t,
$$
by inequalities (\ref{id1}) and (\ref{cava2}) we have
\begin{eqnarray}
\nonumber
t(f^{**}(t)-f^{*}(t))
&=&
\int_{f^*(t)}^{\infty}d_{f,\mu}(\lambda)d\lambda
\\
\nonumber
&=&
\int_{\{x\in X:|f(x)|>f^*(t)\}}|f|d\mu-f^*(t)d_{f,\mu}(f^*(t))\\
\nonumber
&\leq&
Md_{f,\mu}(f^*(t))\leq Mt.
\end{eqnarray}

This proves the equivalence of the first two conditions.

The conditions (ii) and (iii) are equivalent by the formula (\ref{cava2}). 

Assume then that the condition (iii) is valid.
The condition (iv) is a consequence of a general real analysis principle. Indeed, given any decreasing function $g$ that satisfies
$$
\int_{s}^{\infty}g(t)dt \leq Mg(s)
$$
for every $s>\alpha$, we have
$$
g(s+t)\leq c_1g(s)e^{-c_2t/M}
$$
for every $s>\alpha$ and $t>0$ with $c_1=4$ and $c_2=\log 2$. To see this, observe that
since $g$ is decreasing, the integral condition can be rewritten as
$$
\sum_{i=1}^{\infty}g(s+iM) \leq g(s)
$$
for all $s> \alpha$. Using the condition recursively, we have
$$
g(s) \geq 
2\sum_{i=2}^{\infty}g(s+iM) \geq
2^{k-1}g(s+kM)
$$
for all $s>\alpha$ and every positive integer $k$. Since there
exists a smallest positive integer so that $t< kM$, we may apply the
preceding inequality to obtain 
\begin{equation}
\label{ineqg}
g(s+t) \leq g(s+(k-1)M) \leq 2^{2-k}g(s) \leq 4 e^{-\log 2t/M}g(s).
\end{equation}
The condition (iv) follows with $g$ replaced by the distribution
function of $f$. 

Assume now the condition (iv). Then
\begin{eqnarray}
\nonumber
\int_{\{x\in X:|f(x)|>\lambda\}}|f|d\mu
&=&
\int_{\lambda}^{\infty}d_{f,\mu}(s)ds + \lambda d_{f,\mu}(\lambda)
\\
\nonumber
&\leq&
\int_{\lambda}^{\infty}c_1d_{f,\mu}(\lambda)e^{c_2(\lambda-s)}ds 
+ \lambda d_{f,\mu}(\lambda)
\\
\nonumber
&\leq&
\left(\frac{c_1}{c_2}+\lambda\right)d_{f,\mu}(\lambda).
\end{eqnarray}
Hence we have the second condition with $M=c_1/c_2$.
\end{proof}

The previous theorem provides us with interesting knowledge on the
behaviour of the functions in $L^{\infty}_w(X)$. Indeed, there are no
big gaps in the distribution function of $f$. In terms of decreasing
rearrangements we see that given $f$ and $M$ as in the second
condition, $f^*$ is continuous except for a countable set of points
where the size of the jump is at most $M$. In addition, the mass is always concentrated on the low level sets. 
\begin{corollary}
Let $f\in L^{\infty}_w(X)$. Suppose that $d_{f,\mu}(\lambda)$ is positive and
finite. Then for any $\gamma >1$ we have
$$
\frac{\mu \left(\{ x\in X: \lambda < |f(x)| \leq \lambda + \gamma
    M\}\right)}
{\mu \left(\{ x\in X: \lambda < |f(x)| \}\right)} 
\geq
(1-2^{2-\gamma}),
$$
where $M$ is the smallest constant satisfying the second condition in
Theorem \ref{theorem1}.
\end{corollary}
\begin{proof}
By Theorem \ref{theorem1}(ii) we have
$$
\int_{\{x\in X:|f(x)|>\lambda\}}|f|d\mu \leq (\lambda +
M)d_{f,\mu}(\lambda).
$$
Inequality
(\ref{ineqg}) in the proof of Theorem \ref{theorem1} implies
$$
d_{f,\mu}(\lambda+\gamma M) \leq d_{f,\mu}(\lambda)2^{2-\gamma}
$$
and the claim follows.
\end{proof}

Sometimes it is possible to calculate precisely the integral average
of a function over its level sets. In the following example we have a
function which is extremal for the second condition of Theorem
\ref{theorem1}.

\begin{example}
Let $f:\Rn \to \R$ with $f(x)=\log (|x|^{-n})\chi_{B(0,1)}$, 
when $x\neq 0$ and $f(0)=0$. 
Then $f$ belongs to
$L^{\infty}_w(\Rn)$, since
$$
\int_{\{x\in \Rn:|f(x)|>\lambda\}}|f(y)|dy 
= 
(\lambda +1)
|\{x\in \Rn:|f(x)|>\lambda\}|
$$
for all $\lambda \geq 0$. 
\end{example}

\section{Functions of bounded mean tail oscillation}

In this section we localize the definition of the weak $L^{\infty}(X)$.
Let $(X,d,\mu)$ be a metric measure space.
A ball with radius $r>0$ and center $x$ is denoted by
$$
B(x,r)=\{ y\in X: d(y,x) < r\}
$$
Let $f$ be a real-valued function defined in
$(X, d, \mu)$. 
We write
$$
f_B = \mvint_Bfd\mu = \frac{1}{\mu(B)}\int_Bfd\mu
$$
for the mean value integral over the  ball $B$. 
If $\mu(B)=0$, then we set $f_B=0$.
A locally integrable function $f$ is of bounded mean oscillation, if
there exists $M\geq 0$ so that
$$
\mvint_B|f-f_B|d\mu \leq M
$$
for every ball $B\subset X$ and we write $f\in BMO$.
Similarly, if there exists $M\geq 0$ so that 
$$
\int_{\{x\in B:|f(x)-f_B|>\lambda\} }|f-f_B|d\mu
\leq
(\lambda+M)\mu(\{x\in B:|f(x)-f_B|>\lambda\})
$$
for all $\lambda>0$ and for all  balls $B\subset X$, 
we say that $f$ is of bounded mean tail oscillation and write
$f\in BMTO(X)$. Functions of bounded mean tail oscillation are of
bounded mean oscillation and satisfy the John-Nirenberg inequality.
\begin{theorem}
\label{abmokar}
Let $f\in L^1_{loc}(X)$. Then the following conditions are equivalent:
\begin{itemize}
\item[(i)]  There exists $M\geq 0$ so that
\begin{align}
&\int_{\{x\in B: |f(x)-f_B|>\lambda \}}|f-f_B|d\mu 
\nonumber\\ \nonumber
&\qquad\leq 
(\lambda+M)\mu(\{x\in B:|f(x)-f_B|>\lambda\})
\end{align}
for all $\lambda \geq 0$ and for all  balls $B\subset X$. 
\item[(ii)] There exist constants $c_1,c_2>0$ so that
\begin{align}
&\mu(\{x\in B:|f(x)-f_B|>\lambda_2\})
\nonumber\\
\nonumber
&\qquad \leq
c_1\mu(\{x\in B:|f(x)-f_B|>\lambda_1\})e^{c_2(\lambda_2-\lambda_1)}
\end{align}
for all $0\leq \lambda_1<\lambda_2$ and for all  balls $B\subset X$.
\end{itemize}
In addition, if $f$ is of bounded mean tail oscillation, then
$f\in BMO(X)$ with
$$
\sup _B \mvint_{B}|f-f_B|d\mu \leq M
$$
and $f$ satisfies
the John-Nirenberg inequality: there exist $c_1,c_2>0$ so that
$$
\mu(\{x\in B: |f-f_B|>\lambda\})
\leq
c_1\mu(B)e^{-c_2\frac{\lambda}{M}}
$$
for all $\lambda>0$.
\end{theorem}

\begin{proof}
Suppose that $f\in BMTO(X)$. Fix a ball $B$ and
consider the measure space $(B,\mu|_B)$. Then
$f-f_B\in L^{\infty}_w(B,\mu|_B)$ with $\alpha =0$ and we may
apply Theorem \ref{theorem1}. Hence, by inequality (\ref{ineqg}), we have
\begin{equation}
\nonumber
\mu(\{x\in B:|f(x)-f_B|>\lambda _2\})
\leq
4\mu(\{x\in B:|f(x)-f_B|>\lambda _1\})2^{(\lambda_1-\lambda_2)/M}.
\end{equation}
Since $B$ is arbitrary, the second condition follows.

Assume then that $f$ satisfies the second condition.
Applying Theorem \ref{theorem1} to measure space
$(B,\mu|_B)$ for every  ball $B\subset X$ we get the first
condition with $M = c_1/c_2$.

For the $BMO$ condition we use the  $BMTO$ condition with
$\lambda =0$. The John-Nirenberg inequality follows at once.
\end{proof}

\section{$BMO$ and the weak $L^{\infty}$ for doubling measures}

In this section the focus is in the connection between $BMO$ and
$L^{\infty}_w$. We assume that the measure space is doubling which
guarantees covering properties of the space. These can be used to
prove that the space $BMO$ is included in $L^{\infty}_w$. We also
characterize essentially bounded functions and 
functions of bounded mean oscillation
with a condition similar to the $BMTO$ condition.

\subsection{Doubling measures}
Let $(X,d,\mu)$ be a metric space endowed with a metric $d$ and a
Borel regular measure $\mu$ so that all open balls
have positive and finite measure. During the rest of
the current section we assume that the measure $\mu$ is doubling, i.e.
there exists a constant $c_{\mu}\geq 1$, called the doubling constant
of $\mu$, so that
$$
\mu (B(x,2r))\leq c_{\mu}\mu(B(x,r))
$$
for all $x\in X$ and $r>0$.

The doubling condition implies a covering theorem.
Indeed, given any collection of balls
with uniformly bounded radius, there exists a pairwise disjoint,
countable subcollection of balls, whose 5-dilates cover the union of
the original collection. This theorem implies Lebesgue's
differentiation theorem, which guarantees that any locally integrable
function can be approximated at almost every point by integral
averages of the function over a contracting sequence of balls. For the
proofs we refer to \cite{CW} and \cite{H}.

\subsection{Essentially bounded functions}

Let us study the localised weak $L^{\infty}$.
Here we show that essentially bounded
functions can be
characterized as functions satisfying the weak $L^{\infty}$ condition
in every ball of the space. 
\begin{lemma}
\label{Linfty}
Let $f$ be a measurable function in $X$. 
Then $f\in L^{\infty}(X)$ if and only if there exists $M\geq 0$
so that
$$
\int_{\{x\in B: |f(x)|>\lambda \}}|f|d\mu 
\leq (\lambda + M)\mu(\{x\in B:|f(x)|>\lambda\}),
$$
for all  balls $B\subset X$ and for all $\lambda \geq 0$. 
\end{lemma}
\begin{proof}

Let $f\in L^{\infty}(X)$ and $B\subset X$ an arbitrary ball. Then
$$
\int_{\{x\in B: |f(x)|>\lambda\}}|f|d\mu
\leq
\|f\|_{L^{\infty}(X)}\mu(\{x\in B: |f(x)|>\lambda\})
$$
and the necessity of the condition follows.

For the sufficiency, 
let $f$ be a function on $X$ satisfying
$$
\int_{\{x\in B: |f(x)|>\lambda \}}|f|d\mu 
\leq (\lambda + M)\mu(\{x\in B: |f(x)|>\lambda\}),
$$
for all balls $B\subset X$ and for all $\lambda \geq 0$ with some
$M\geq 0$. Define $A =\{x\in X: |f(x)|>2M\}$. Suppose
$\mu(A)>0$. By Lebesgue's differentiation theorem, $A$ has at least
one density point, say $a\in A$. Hence, there exists a ball
$B(a,r) \subset X$ so that
$$
\frac{\mu (B(a,r)\cap A)}{\mu(B(a,r))} > \frac{1}{2}.
$$
On the other hand
\begin{eqnarray}
\nonumber
2M\mu(B(a,r)\cap A)
&\leq&
\int_{B(a,r)\cap A}|f|d\mu
\\
\nonumber
&\leq&
\int_{\{x\in B(a,r):|f(x)|>0\}}|f|d\mu
\\
\nonumber
&\leq&
(0 + M)\mu(B(a,r)).
\end{eqnarray}
This is a contradiction and hence the proof is complete.
\end{proof}

\subsection{A covering lemma}

Here we present a Calder\'on-Zygmund type covering lemma which is a
generalization of the Lemma 3.2 in \cite{BDS}.
\begin{lemma}
\label{coveringlemma}
Let $B_0$ be a ball. 
Let $F\subset 3B_0$ be a
measurable set and denote
$E=F\cap B_0$. If $\mu (F)\leq \frac{1}{2}\mu(B_0)$,
then there exists a disjoint family of balls contained in $3B_0$ so that
\begin{eqnarray}
\nonumber
&(i)&\mu(5B_i\cap F) \leq \mu(5B_i\setminus F),\qquad \qquad \\
\nonumber
&(ii)&\mu\left(E\setminus \bigcup_{i=1}^{\infty}5B_i\right)=0,\\
\nonumber
&(iii)&\sum_{i=1}^{\infty} \mu (5B_i) \leq c \mu (F),
\end{eqnarray}
where $c$ depends only on the doubling constant of the measure $\mu$. 
\end{lemma}

\begin{proof}
Let $x\in E$ be a density point of $E$. Since
$$
\mu(F\cap B(x,2r_0))
\leq
\mu(F)
\leq
\frac{1}{2}\mu(B_0)
\leq
\frac{1}{2}\mu (B(x,2r_0)),
$$
we have
$$
\mu(B(x,2r_0)\cap F) \leq \mu (B(x,2r_0)\setminus F)
$$
and since $x$ is a density point of $E$,
there exists a greatest integer $k$ so that $5B_x=B(x,2^{1-k}r_0)$
satisfies
$$
\mu(5B_x\cap F)\leq \mu(5B_x\setminus F).
$$
By the maximality of $k$ we also get
$$
\mu(B(x,2^{1-j}r_0)\cap F)> \mu(B(x,2^{1-j}r_0)\setminus F)
$$
for all $j > k$.
By a covering theorem for the balls $B_x$, there exists a
countable family of balls $\{B_i\}$ which satisfy the first condition
of the lemma. Since almost every point is a density point the
condition (ii) follows. For the
condition (iii), we observe that
\begin{eqnarray}
\sum_{i=1}^{\infty}\mu(5B_i)
\nonumber
&\leq&
c_{\mu}^3\sum_{i=1}^{\infty}\mu\left(\frac{5}{8}B_i\right)\\
\nonumber
&\leq&
2c_{\mu}^3\sum_{i=1}^{\infty}\mu\left(\frac{5}{8}B_i \cap F\right)
\\
\nonumber
&\leq&
2c_{\mu}^3\mu(F).
\end{eqnarray}
\end{proof}

\subsection{Functions of bounded mean oscillation}

In doubling metric measure space it is possible to characterize the
functions of bounded mean oscillation with a condition similar to that
of bounded mean tail oscillation. 
\begin{theorem}
\label{doublingcoveringlemma}
Let $f\in L^1_{loc}(X)$. Then
$f\in BMO(X)$ if and only if there exists $M\geq 0$ so that
$$
\int_{\{ x\in B: |f(x)-f_B| > \lambda\} }|f-f_B|d\mu
\leq
(\lambda + M)\mu (\{ x\in 3B: |f(x)-f_B| > \lambda\} )
$$
for all $\lambda \geq 0$ and for all balls $B\subset X$.
\end{theorem}
\begin{proof}
Assume $f\in BMO(X)$ and write $\|f\|_*$ for the
$BMO$-norm of the function $f$.
Fix a ball $B$. Given $\lambda >0$ we write
$$
E_{\lambda}=\{ x\in B: |f(x)-f_B|>\lambda\}
$$
and
$$
F_{\lambda}=\{ x\in 3B: |f(x)-f_B|>\lambda\}.
$$
We have
$$
\mvint_{3B}|f-f_B|d\mu
\leq
\mvint_{3B}|f-f_{3B}|d\mu
+
|f_B-f_{3B}|
\leq
(1+c_{\mu}^2)\|f\|_*.
$$
Hence, 
$$
\mu(F_{\lambda})
\leq 
\int_{3B}\frac{|f-f_B|}{\lambda}d\mu
\leq
\frac{1}{2}\mu(B)
$$
whenever $\lambda \geq \lambda_0=2c_{\mu}^2(1+c_{\mu}^2)\|f\|_*$.
Fix any $\lambda > \lambda _0$.
Now we apply Lemma \ref{coveringlemma} to the sets $E_{\lambda}$ and
$F_{\lambda}$. We have
\begin{eqnarray}
\nonumber
\int_{E_{\lambda}}(|f-f_B|-\lambda )d\mu
&\leq&
\sum_{i=1}^{\infty}\int_{5B_i\cap F_{\lambda}}(|f-f_B|-\lambda)d\mu 
\\
\nonumber
&\leq&
\sum_{i=1}^{\infty}\int_{5B_i\cap F_{\lambda}}|f-f_{5B_i}|d\mu 
\\
\nonumber
& &
+ \sum_{i=1}^{\infty}\mu({5B_i\cap F_{\lambda}})(|(f-f_B)_{5B_i}|-\lambda).
\end{eqnarray}
Let $P$ be the set of indices for which the last sum has a
positive term. Then
\begin{eqnarray}
\nonumber
\sum_{i=1}^{\infty}\mu({5B_i\cap F_{\lambda}})(|(f-f_B)_{5B_i}|-\lambda)
&\leq &
\sum_{i\in P}\mu(
{5B_i \setminus F_{\lambda}}
)(|f_{B}-f_{5B_i}|-\lambda)
\\
\nonumber
&\leq&
\sum_{i=1}^{\infty}\int_{{5B_i \setminus F_{\lambda}}}(|f_{B}-f_{5B_i}|-|f-f_B|)d\mu
\\
\nonumber
&\leq&
\sum_{i=1}^{\infty}\int_{{5B_i \setminus F_{\lambda}}}|f-f_{5B_i}|d\mu .
\end{eqnarray}
This gives
\begin{eqnarray}
\int_{E_{\lambda}}(|f-f_B|-\lambda )d\mu
\nonumber
&\leq&
\sum_i\int_{5B_i}|f-f_{5B_i}|d\mu.
\\
\nonumber
&\leq&
\|f\|_*\sum_i\mu(5B_i)
\\
\nonumber
&\leq&
2c\|f\|_*\mu(F_{\lambda}).
\end{eqnarray}
The last inequality follows from 
Lemma \ref{coveringlemma}.
For $\lambda < \lambda _0$ we have
\begin{eqnarray}
\nonumber
\int_{E_{\lambda}}|f-f_B|d\mu
&=&
\int_{E_{\lambda _0}}|f-f_B|d\mu
+
\int_{E_{\lambda}\setminus E_{\lambda _0}}|f-f_B|d\mu
\\
\nonumber
&\leq&
(\lambda _0+c\|f\|_*)\mu (F_{\lambda_0})
+\lambda_0(\mu(E_{\lambda})-\mu(E_{\lambda _0}))
\\
\nonumber
&\leq&
(\lambda+2\lambda_0+c\|f\|_*)\mu(F_{\lambda}).
\end{eqnarray}
Now the theorem follows.
\end{proof}

\begin{remark}
If we assume
$$
\mu(F)\leq \frac{\mu (B_0)}{2c_{\mu}^k},
$$
then we have 
$$
\bigcup 5B_i \subset (1+2^{-k})B_0
$$ 
in the covering lemma.
Consequently, we have another version of the Theorem
\ref{doublingcoveringlemma}. Indeed, let $\rho >1$. Then
$f\in BMO(X)$ if and only if there exists $M\geq 0$ so that
$$
\int_{\{ x\in B: |f(x)-f_B| > \lambda\} }|f-f_B|d\mu
\leq
(\lambda + M)\mu (\{ x\in \rho B: |f(x)-f_B| > \lambda\} ).
$$
Rewriting the proof of the lemma with $\rho>1$ instead of constant $3$
shows that the bound $M$ in the final estimate blows up when $\rho$
approaches 1.
\end{remark}

\begin{remark}
In some metric spaces the above result can be sharpened. 
Indeed, if
every pair of points can be joined by a curve with a length as close
to their distance as wished, a similar argument as above shows that
$f\in BMO(X)$ if and only if there exists $M\geq 0$ so that
$$
\int_{\{ x\in B: |f(x)-f_B| > \lambda\} }|f-f_B|d\mu
\leq
(\lambda + M)\mu (\{ x\in  B: |f(x)-f_B| > \lambda\} ).
$$
In particular, this characterization is valid for doubling measures in $\Rn$.
\end{remark}

\subsection{The Weak $L^{\infty}$ and $BMO$}

The following theorem establishes the connection between the weak
$L^{\infty}(X)$ and $BMO(X)$. Briefly, every function of bounded mean
oscillation with a finite distribution function belongs to the weak
$L^{\infty}(X)$.

\begin{theorem}
\label{weakbmodoubling}
If $f\in BMO(X)$ and there exists $\alpha \geq 0$ so that
$d_{f,\mu}(\lambda)$ is finite for all $\lambda > \alpha$, then there
exists $M\geq0$ so that 
$$
\int_{\{x\in X:|f(x)|>\lambda\}}|f|d\mu \leq (\lambda + M)
\mu(\{x\in X:|f(x)|>\lambda\})
$$
for all $\lambda > \alpha$.
\end{theorem}
\begin{proof}
Let $f\in BMO(X)$. We may assume $f$ is positive since $|f|$ is also
in $BMO(X)$ with a norm at most twice that of $f$. 

We split the proof in two cases according to the total mass of the space
$X$. First we suppose $\mu (X)=\infty$. Let $\epsilon >0$ and consider
$\lambda \geq \alpha + \epsilon$, 
where $\alpha\geq 0$ is given by the assumption. Fix a point $x_0 \in
X$ and define
$$
E_k^{\lambda} = \{ x\in B(x_0,k): f(x)>\lambda \}
$$
and $F_k^{\lambda}=E_{3k}^{\lambda}$.
Since $d_{f,\mu}(\alpha + \epsilon)$ is finite,
$$
\lim_{k\to \infty}\frac{\mu(F_k^{\lambda})}{\mu(B(x_0,k))} = 0
$$
and consequently there exists an integer $N_{\epsilon}$ so that the
above quotion is at
most $\frac{1}{2}$ for any $\lambda \geq \alpha + \epsilon$ whenever
$k>N_{\epsilon}$. Fix some
$k>N_{\epsilon}$,
apply Lemma \ref{coveringlemma} to the set
$F_k^{\lambda}$ and obtain a collection of balls $\{B_i\}$. This implies
\begin{eqnarray}
\nonumber
& &
\int_{E_k^{\lambda}}(f-\lambda)d\mu
\\
\nonumber
&\leq&
\sum_{i=1}^{\infty}
\left(
\int_{5B_i\cap F_k^{\lambda}}|f-f_{5B_i}|d\mu
+
\mu(5B_i\cap F_k^{\lambda})(f_{5B_i}-\lambda)
\right).
\end{eqnarray}
Similar to the proof of Theorem \ref{doublingcoveringlemma} we
estimate the last term
in the summand with
$$
\mu(5B_i\cap F_k^{\lambda})(f_{5B_i}-\lambda)
\leq
\int_{5B_i\setminus F_k^{\lambda}}|f-f_{5B_i}|d\mu
$$
and hence
$$
\int_{E_k^{\lambda}}(f-\lambda)d\mu
\leq
\sum_{i=1}^{\infty}
\int_{5B_i}|f-f_{5B_i}|d\mu
\leq
c\|f\|_*\mu(F_k^{\lambda}).
$$
By monotone convergence theorem the above inequality may
be passed to the limit (w.r.t. $k$) 
and the theorem follows since the conclusion is
independent of $\epsilon$.

Let us then consider the case $\mu(X)<\infty$. Then $\alpha$ may be
discarded since all sets in $X$ have finite measure. Write
$$
E^{\lambda} = \{x\in X: f(x)>\lambda\}.
$$
Since $f$ is finite $\mu$ almost everywhere, $\mu(E^{\lambda})$
approaches zero as $\lambda$ tends to infinity. Hence we may pick a
$\lambda _0$ so that $\mu(E^{\lambda})\leq \frac{1}{4}\mu(X)$ for any
$\lambda \geq \lambda _0$.
Since $\mu (B(x_0,k))\to \mu(X)$ as $k$ tends to infinity, we have
$N>0$ so that 
$$
\frac{\mu(E_k^{\lambda})}{\mu(B(x_0,k))}\leq \frac{1}{2}
$$
for all $\lambda \geq \lambda_0$ and $k> N$. As above we apply
Lemma \ref{coveringlemma} and monotone convergence theorem. We have
$$
\int_{E^{\lambda}}(f-\lambda)\leq c\|f\|_*d_{f,\mu}(\lambda)
$$
for all $\lambda \geq \lambda_0$. The case of small $\lambda$ is
treated as in the end of the proof of the 
Theorem \ref{doublingcoveringlemma}.
\end{proof}

The preceding theorem shows that any function of bounded mean
oscillation belongs to the $L^{\infty}_w$ as soon as the the measure
space is doubling. The converse is not true in general since the local
behaviour of the function in $L^{\infty}_w$ is not controlled for.

In a Euclidean cube, equipped with the Lebesgue measure, there is
a converse result stating that 
every function in the weak $L^{\infty}$ is
equimeasurable to some function of bounded mean oscillation
\cite{BDS}. 
However, if the
measure is changed, the result is no longer true as the next
example shows.
\begin{lemma}
There exists a doubling metric measure space $(X,\mu)$ and a function
$f\in L^{\infty}_w(X)$ so that no function $g$ defined on $X$,
equimeasurable to $f$, belongs to $BMO(X)$.
\end{lemma}

\begin{proof}
We define $X\subset \R$ to be the countable collection of points
$$
X=\{x_k\}_{k=0}^{\infty}=\{2^{-k}\}_{k=0}^{\infty}
$$
with the Euclidean distance. We set
$$
\mu(x_k)=2^{-k}
$$
for every $x_k\in X$. The metric measure space $(X,\mu)$ is doubling
with a doubling constant $c_{\mu}=4$ and $\mu(X)=2$. Let us now define
$$
f(x_k)= (-1)^kk.
$$
and fix $\lambda>0$.
Then
$$
\int_{\{x\in X:|f(x)|>\lambda\}}|f|d\mu
\leq
(\lambda+2)\mu(\{x\in X:|f(x)|>\lambda\}).
$$
and hence $f \in L^{\infty}(X)$. Nevertheless, $f$ is not of bounded
mean oscillation. This can be seen by considering the balls 
$$
B_k=B\left(x_{2k},5\cdot 2^{-2k-3}\right)
$$
since now
$$
\mvint_{B_k}|f-f_{B_k}|d\mu = \frac{16k+10}{27}
$$
which blows up as $k$ grows.
Observe that because of the special structure of the space, $f$ is the only function defined on $X$ which is equimeasurable to $f$. This proves the lemma.
\end{proof}

\begin{remark}
The Hardy-Littlewood maximal operater is bounded
in $BMO$ for doubling metric measure spaces \cite{AK}, for Euclidean
case see \cite{BDS} and \cite{CF}. The argument in \cite{BDS}
showing that maximal operator preserves the weak $L^{\infty}$ in the 
Euclidean case, depend only
on the weak and strong type estimates of the maximal operator.
Since these estimates are available also in doubling metric measure
spaces, we can conclude that the maximal operator is bounded in
$L^{\infty}_w(X)$ as well.
\end{remark}

\section{$BMO$ and the weak $L^{\infty}$ for non-doubling measures}

In this section we show that the arguments of the previous section
can be generalized to some non-doubling measures in
Euclidean spaces.
Indeed, we study positive Radon measures $\mu$ for which no
hyperplane $L$, orthogonal to one of the coordinate axes, contains
mass. These measures are not rare since for every nonnegative Radon
measure for which $\mu(p)=0$ at every point $p\in \Rn$, there exists
an orthonormal system so that the above mentioned hyperplane condition
is satisfied (for further details, see \cite{MMNO}). In this case Lemma
\ref{coveringlemma} is replaced by the following result.
\begin{lemma}
Let $\mu$ be a positive Radon measure in $\Rn$ such that for every
hyperplane $L$, orthogonal to one of the coordinate axes, $\mu(L)=0$.
Let $E$ be a subset of $\Rn$. Suppose $E$ is contained in a cube $Q_0$
with sides parallel to the coordinate axes, and suppose that 
$
\mu(E)\leq \frac{1}{2}\mu(Q_0).
$
Then there exists a sequence $\{Q_i\}$ of cubes with sides parallel to
the coordinate axes and contained in $Q_0$ such that 
\begin{eqnarray}
&(i)&\mu(Q_i\cap E)\leq \mu(Q_i\setminus E)\nonumber \qquad \qquad\\
&(ii)&\mu\left(E\setminus \bigcup_{i=1}^{\infty}Q_i\right)=0, \nonumber \\
\nonumber
&(iii)&\sum_{i=1}^{\infty} \mu (Q_i) \leq c(n) \mu (E),
\end{eqnarray}
with $c(n)$ depending only on the dimension of $\Rn$.
\end{lemma}

Following exactly the same line of arguments as in the proofs of
Theorems \ref{doublingcoveringlemma} and \ref{weakbmodoubling} we
obtain the following results. 
\begin{theorem}
Let $\mu$ be a positive Radon measure in $\Rn$ such that for every
hyperplane $L$, orthogonal to the coordinate axes, $\mu (L)=0$. Then
$f\in BMO(\mu)$ if and only if there exists $M\geq 0$ such that
$$
\int_{\{x\in Q:|f(x)-f_Q|>\lambda\}}|f-f_Q|d\mu 
\leq 
(\lambda + M)
\mu(\{
x\in Q:|f(x)-f_Q|>\lambda
\})
$$
for all cubes $Q$ and $\lambda \geq 0$.
\end{theorem}

\begin{theorem}
Let $\mu$ be a positive Radon measure in $\Rn$ such that for every
hyperplane $L$, orthogonal to the coordinate axes, $\mu (L)=0$.
If $f\in BMO(\mu)$ and there exists $\alpha \geq 0$ so that
$d_{f,\mu}(\lambda)$ is finite for all $\lambda > \alpha$, 
then $f\in L^{\infty}_w(\mu)$. 
\end{theorem}

Whether there exists a converse, i.e. if for every $f\in
L^{\infty}_w(\mu)$ there exists an equimeasurable function of bounded
mean oscillation, is not clear. Observe that the counterexample in the
previous section was singular with respect to Lebesgue measure 
and does not satisfy the hyperplane condition.

\bibliographystyle{acm}
\bibliography{newbiblio}

\vspace{0.5cm}
\noindent
\small{\textsc{Department of Mathematics},}
\small{\textsc{FI-20014 University of Turku},}
\small{\textsc{Finland}}\\
\footnotesize{\texttt{daniel.aalto@iki.fi}}

\end{document}